\documentclass[11pt]{amsart}

\usepackage[T1]{fontenc}
\usepackage{lmodern}
\usepackage{geometry}
\usepackage{comment}
\usepackage{amsmath}
\usepackage{amssymb}
\usepackage{amsfonts}
\usepackage{mathrsfs}
\usepackage{enumitem}
\usepackage{graphicx}

\theoremstyle{plain}
\newtheorem{theorem}{Theorem}[section]
\newtheorem{lemma}[theorem]{Lemma}

\theoremstyle{definition}
\newtheorem{definition}[theorem]{Definition}

\theoremstyle{remark}

\title[On the weak operator Daugavet property and poor subspaces]
{On the weak operator Daugavet property and poor subspaces}

\author{Samir Hamad}

\email{samih49@zedat.fu-berlin.de}

\date{August 2026}


\begin{document}

\begin{abstract}
We will provide an example of a Banach space with the Daugavet property which does not possess the weak operator Daugavet property. We will also show that every separable Banach space with the Daugavet property contains a non-poor subspace isomorphic to $\ell^1$ while noticing that every separable subspace of $L^\infty[0,1]$ is poor.
\end{abstract}

\maketitle

\section{Introduction}
We say that a normed space $X$ has the \emph{Daugavet property}, shortly $X \in \mathrm{DPr}$, if
\[
\|\mathrm{Id} + T\| = 1 + \|T\|,
\]
for every rank-one operator $T \colon X \to X$. By passing to adjoint operators, the following characterization was obtained in \cite[Lemma~2.2]{kadets2000banach}.

\begin{theorem}
Let $X$ be a normed space. The following assertions are equivalent:
\begin{enumerate}
\item $X$ has the Daugavet property.
\item For every $y \in S_X$, $x^* \in S_{X^*}$, and every $\varepsilon > 0$, there exists $x \in S_X$ such that
\[
x^*(x) \ge 1 - \varepsilon
\quad \text{and} \quad
\|x + y\| \ge 2 - \varepsilon.
\]
\item For every $\varepsilon > 0$ and every $y \in S_X$, the closed convex hull of
\[
\{\, u \in (1+\varepsilon) B_X \mid \|y + u\| \ge 2 - \varepsilon \,\}
\]
contains $S_X$.
\end{enumerate}
\end{theorem}

A consequence of this characterization is that every slice of $B_X$, i.e., every set of the form
\[
\operatorname{Slice}(B_X,x^*,\varepsilon) := \{x \in B_X : x^*(x) > 1-\varepsilon\},
\]
where $x^*\in S_{X^*}$ and $\varepsilon>0$, has diameter $2$.

One of the oldest open questions concerning Daugavet spaces asks: if $X$ and $Y$ are Banach spaces with the Daugavet property, do the injective tensor product $X\hat\otimes_\varepsilon Y$ and the projective tensor product $X\hat\otimes_\pi Y$ also have the Daugavet property?
In \cite{WODP2020} the weak operator Daugavet property (WODP) was introduced, a strengthening of the Daugavet property characterized by the existence of certain operators. It was shown there that if $X$ and $Y$ have the WODP, then so does their projective tensor product. For the injective tensor product, it was shown in \cite[Theorem 3.5]{WODP2024} that if the duals $X^*$ and $Y^*$ have the WODP, then $X\hat\otimes_\varepsilon Y$ has the Daugavet property.

Regarding examples, it was observed in \cite[Proposition 5.11]{WODP2020} that $L^1(\mu)$ has the WODP for every $\sigma$-finite atomless measure $\mu$, and that $L^1$-preduals with the Daugavet property have the WODP \cite[Proposition 5.9]{WODP2020}; this class includes $C[0,1]$. We note that a straightforward partition-of-unity argument shows that $C(K,X)$ has the WODP whenever $X$ is a Banach space and $K$ is a compact Hausdorff space without isolated points. Another interesting class of spaces with the WODP consists of $L$-embedded Banach spaces with the metric approximation property and the Daugavet property \cite[Proposition 5.8]{BECERRAGUERRERO2027111632}. The stability of the WODP under (projective) tensor products lends particular significance to the question, raised in \cite{WODP2020}, of whether the Daugavet property and the weak operator Daugavet property in fact coincide. The main result of this paper is the construction of a Banach space with the Daugavet property which does not possess the WODP. After the study of the WODP, we will turn to a short study of poor subspaces, see Definition 3.1, showing that every separable Banach space contains a non-poor subspace isomorphic to $\ell^1$ and showing the necessity of separability, thus answering two questions raised in \cite{quotient}.

\section{Weak operator Daugavet property}

For a normed space $X$ with the Daugavet property, let $S := \operatorname{Slice}(B_X,y^*,\varepsilon)$, where $y^* \in S_{X^*}$, $\varepsilon > 0$, and $y \in S_X$. By the dual characterization of the Daugavet property, we can find $x^* \in S_{X^*}$ with $x^*(y) > 1-\varepsilon$ and $\|x^* + y^*\| > 2-\varepsilon$. Therefore there is some $x \in S$ with $x^*(x) > 1-\varepsilon$. The operator $T(z) := x^*(z)y$ has norm $1$ and
\[
\|Tx - y\| < \varepsilon, \qquad \|Ty - y\| < \varepsilon.
\]
The next definition, which was introduced in \cite{WODP2020}, is a strengthening of the Daugavet property by assuming the existence of a similar operator $T \colon X \to X$ such that $T$ fixes some arbitrary finite set of elements in $S_X$.

\begin{definition}
Let $X$ be a Banach space. We say that $X$ has the \emph{weak operator Daugavet property} (WODP for short) if, given 
$x_1, \dots, x_n \in S_X$, $\varepsilon > 0$, a slice $S$ of $B_X$, and $x' \in B_X$, 
we can find $x \in S$ and an operator $T : X \to X$ such that the following conditions hold:
\begin{enumerate}
    \item $\|T\| \leq 1 + \varepsilon$,
    \item $\|T(x_i) - x_i\| < \varepsilon$ for every $i \in \{1, \dots, n\}$,
    \item $\|T(x) - x'\| < \varepsilon$.
\end{enumerate}
\end{definition}

Looking at the proof that $L^1$-preduals with the Daugavet property have the WODP, we see that the Daugavet property of a Banach space $X$ implies that for a finite-dimensional $E \subset X$, $x' \in S_X$, $\varepsilon > 0$, and a slice $S$, we can find some $x \in S$ and $T \colon \operatorname{span}(E, \{x\}) \to X$ such that $T$ has norm less than $1+\varepsilon$,
\[
\|Te - e\| < \varepsilon \quad \text{for } e \in S_E, \qquad \text{and} \qquad \|Tx - x'\| < \varepsilon.
\]
We shall see that it is not always possible to extend such an operator to the whole of $X$ while almost preserving its norm. Before that, we show that the WODP implies an operator-independent property.
\begin{lemma}
Let $X$ be a Banach space with the WODP, and let $\delta, \varepsilon > 0$. Suppose $x_1, x_2 \in S_X$ are such that $B(x_1, \tfrac{1}{2} ) \cap B(x_2, \tfrac{1}{2}) \neq \varnothing$, and every
\[
z \in B\!\left(x_1,\tfrac{1}{2} + 3/2\varepsilon\right) \cap B\!\left(x_2, \tfrac{1}{2} + 3/2\varepsilon \right),
\]
satisfies $\|z + x_1\| > 2 - \delta$. Then, for every slice $S$ of $B_X$, there exists $s \in S$ such that
\[
(1+\varepsilon)\, \|s + u\| \; > \; 2 - \delta - \varepsilon,
\]
for every $u \in B(x_1, 1/2 ) \cap B(x_2, 1/2 )$.
\end{lemma}

\begin{proof}
Suppose there is some slice $S$ of $B_X$ such that
\[
(1+\varepsilon)\|s + u\| \leq 2 - \delta - \varepsilon \quad \text{for every $s \in S$ and some $u(s) \in B(x_1, 1/2 ) \cap B(x_2, 1/2)$}. \tag{$*$}
\]
According to Definition 2.1, we can choose some $T : X \to X$ and $x \in S$ with:
\begin{enumerate}
    \item $\|T\| \leq 1 + \varepsilon$,
    \item $\|T(x_i) - x_i\| < \varepsilon$ for every $i \in \{1,2\}$,
    \item $\|T(x) - x_1\| < \varepsilon$.
\end{enumerate}
For $z \in B(x_1, 1/2) \cap B(x_2, 1/2 )$, we estimate, for $i = 1,2$,
\begin{align*}
\|T(z) - x_i\|
&\le \|T(z) - T(x_i)\| + \|T(x_i) - x_i\| \\
&\le \|T\|\,\|z - x_i\| + \|T(x_i) - x_i\| \\
&< (1+\varepsilon)\cdot \tfrac{1}{2}  + \varepsilon \\
&= \tfrac{1}{2} + 3/2\varepsilon .
\end{align*}
Hence
\[
T\big(B(x_1,1/2) \cap B(x_2,1/2)\big) \subset B\!\left(x_1, \tfrac{1}{2} + 3/2\varepsilon \right) \cap B\!\left(x_2, \tfrac{1}{2} + 3/2\varepsilon\right).
\]

Therefore by assumption and $(*)$, we have for some $u(x) \in B(x_1, 1/2) \cap B(x_2, 1/2)$
\[
2- \delta - \varepsilon<\mathcal{k}Tu + x_1 \mathcal{k} - \varepsilon \leq \mathcal{k}Tu + Tx \mathcal{k} \leq (1+\varepsilon) \mathcal{k}u + x\mathcal{k} \leq 2 - \delta - \varepsilon.
\]
The contradiction concludes the proof.
\end{proof}

\begin{theorem}
There is a Banach space with the Daugavet property and without the WODP. 
\end{theorem}

\begin{proof}
Let $1/8 >\delta > 0$ be fixed. Denote $A := [0, 1/4]$, $B := [1/3, 2/3]$ and $C := [3/4,1]$. Denote by $Y$ the subspace of $L^\infty[A \cup B \cup C]$ consisting of all
(finite) linear combinations of characteristic functions of intervals
contained in $A \cup B \cup C$, which are, up to a set of Lebesgue measure
zero, locally constant at $1/2$, i.e., constant in a neighborhood of $1/2$ up to a set of Lebesgue measure zero.
Let \((\varepsilon_n)_{n \in \mathbb{N}}\) be a sequence with \(\varepsilon_n > 0\) for all \(n\) and \(\varepsilon_n \to 0\). We will inductively construct a sequence of finite-dimensional subspaces \((X_n)_{n \in \mathbb{N}}\) of \(Y\) satisfying \(X_n \subset X_{n+1}\) for all \(n\) such that for every \(g,f \in S_{X_n}\), there exists a finite family \(\{h_j\}_{j=1}^{m_n} \subset S_{X_{n+1}}\) such that
\[
\left\lVert \sum_{j=1}^{m_n} \lambda_j h_j - f \right\rVert < \varepsilon_n
\quad\text{and}\quad
\|h_j + g\| > 2 - \varepsilon_n
\quad \text{for every } j=1,\dots,m_n,
\tag{1}
\]
for some scalars \(\lambda_1,\dots,\lambda_{m_n} \geq 0\) satisfying
\[
\sum_{j=1}^{m_n} \lambda_j = 1.
\]

Let $f_1 := \chi_{A \cup B}$ and $f_2 := \chi_{B \cup C}$. We will construct the sequence in such a way that, for \(n \geq 2\), the subspace \(X_n\) contains a function
\[
q_n \in B(f_1,1/2) \cap B(f_2,1/2),
\]
such that $q_n$ is constantly equal to $1$ on the interval around $1/2$ with length $1/10n$, while \(q_n\) is equal to \(1/2\) on \(A \cup B \cup C \setminus  [1/2 -1/20n , 1/2 + 1/20n] \). Finally, we will require that, whenever \(u \in X_n\) satisfies
\[
u(A \cup C) \subset [1/2-\delta,1/2+ \delta],
\]
then
\[
u \geq 1-2 \delta \text{ in a neighborhood of $1/2$}. \tag{2}
\]
 We begin with \(X_1 := \operatorname{span}(f_1,f_2) \subset Y\). It is easy to see that all requirements are fulfilled for $X_1$. Indeed, for $(\lambda_1 f_1 + \lambda_2 f_2)(A \cup C) \subset [1/2-\delta,1/2+\delta]$, we have necessarily $\lambda_1, \lambda_2 \geq 1/2 - \delta$. Suppose now that we have constructed \(X_1,\dots,X_n \subset Y\).

Let $p_1,\dots,p_w$ be an $\varepsilon_n/2$-net of $S_{X_n}$. Since $X_n$ is finite-dimensional,
we can decompose
\[
A \cup B \cup C = \bigcup_{i=1}^{m} I_i^n,
\]
where the intervals $I_i^n$ intersect only at their endpoints, chosen so that
every element of $X_n$ is a linear combination of the characteristic functions
$\chi_{I_i^n}$.

Since every function in $X_n$ is locally constant at $1/2$, we may choose a
compact interval $J_n$ around $1/2$ with
\[
\mathcal{L}(J_n) < 4\,\mathcal{L}(I_i^n) \qquad \text{for every } i,
\]
where $\mathcal{L}$ denotes the one-dimensional Lebesgue measure, so that
\[
\big\| f|_{(A\cup B\cup C)\setminus J_n} \big\|_\infty = 1
\qquad \text{for every } f \in S_{X_n}.
\]

We can choose $\left\lfloor 4/\varepsilon_n \right\rfloor + 1$ many compact, pairwise disjoint intervals $\{Q^i_j\}_j$ and again $\left\lfloor 4/\varepsilon_n \right\rfloor + 1$ many compact, pairwise disjoint intervals $\{W^i_j\}_j$, with $\bigcup_j Q^i_j$ and $\bigcup_j W^i_j$ disjoint, such that $Q^i_j$ and $W^i_j$ are disjoint from $J_n$ and contained in the right half of $I^n_i$ for every $j$. We can then choose scalars $\lambda'_1, \dots, \lambda'_m$ and $\mu'_1, \dots, \mu'_m$, each with $|\lambda'_i|, |\mu'_i| \leq 2$, such that
\[
p_1 + \sum_{i=1}^{m} \lambda'_i \, \chi_{Q^i_j} + \sum_{i=1}^{m} \mu'_i \, \chi_{W^i_j} \in S_Y \qquad \text{for every } j,
\]
while
\[
\left\| p_1 + \sum_{i=1}^{m} \lambda'_i \, \chi_{Q^i_j} + \sum_{i=1}^{m} \mu'_i \, \chi_{W^i_j} + p_s \right\| = 2 \qquad \text{for every } s \text{ and } j.
\]
Indeed, since every element in $X_n$ is constant on every interval $I^n_i$ — in particular $p_1 \in X_n$ — we may choose $\lambda'_i$ in such a way that $p_1 + \lambda'_i \chi_{Q^i_j}$ is equal to $1$ on $Q^i_j$, while we can choose $\mu'_i$ in such a way that $p_1 + \mu'_i \chi_{W^i_j}$ is equal to $-1$ on $W^i_j$. Since all the intervals $Q^i_j$ and $W^i_j$ are pairwise disjoint, we get that
$p_1 + \sum_{i=1}^{m} \lambda'_i \, \chi_{Q^i_j} + \sum_{i=1}^{m} \mu'_i \, \chi_{W^i_j}$ attains the values $1$ and $-1$ on a set of positive measure on every interval $I^n_i$. This implies that the sum is even codirected with every element in $X_n$, i.e., 
\[
\mathcal{k}p_1 + \sum_{i=1}^{m} \lambda'_i \, \chi_{Q^i_j} + \sum_{i=1}^{m} \mu'_i \, \chi_{W^i_j} + h \mathcal{k} = 2,
\]
for every element $h \in S_{X_n}$, since they are constant on each $I^n_i$.
It is easy to see that
\[
\left\|
 p_1 \;-\; \frac{1}{\left\lfloor 4/\varepsilon_n \right\rfloor + 1}
\sum_{j=1}^{\left\lfloor 4/\varepsilon_n \right\rfloor + 1}
\left(
p_1 + \sum_{i=1}^{m} \lambda'_i \, \chi_{Q^i_j} + \sum_{i=1}^{m} \mu'_i \, \chi_{W^i_j}
\right)
\right\|
\;<\; \frac{\varepsilon_n}{2},
\]
since the sets in $\{Q^i_j\} \cup \{W^i_j\}$ are pairwise disjoint. By construction, (2) is preserved for $\operatorname{span}\big(X_n, \{\chi_{Q^i_j}\}_{i,j}, \{\chi_{W^i_j}\}_{i,j}\big) \subset Y$: let $s^i_j, t^i_j \in \mathbb{R}$ be arbitrary scalars. For $f \in X_n$ with
\[
f(A \cup C) \not\subset \left[\tfrac{1}{2} - \delta,\ \tfrac{1}{2} + \delta\right],
\]
we still have
\[
\left(f + \sum_{i,j} s^i_j \chi_{Q^i_j} + \sum_{i,j} t^i_j \chi_{W^i_j}\right)(A \cup C) \not\subset \left[\tfrac{1}{2} - \delta,\ \tfrac{1}{2} + \delta\right].
\]
Indeed, note first that, since $f$ is constant on every $I^n_i$, there is some interval $I^n_q \subset A \cup C$ with $f(I^n_q)$ disjoint from $[1/2-\delta,1/2+\delta]$. Since $\sum_{i,j} s^i_j \chi_{Q^i_j} + \sum_{i,j} t^i_j \chi_{W^i_j}$ is supported only on the right half of each $I^n_i$, we see that
\[
\left(f + \sum_{i,j} s^i_j \chi_{Q^i_j} + \sum_{i,j} t^i_j \chi_{W^i_j}\right)(x) \notin [1/2-\delta,1/2 + \delta],
\]
for all $x$ on the left half of $I^n_q$. For $f$ with
\[
f(A \cup C) \subset \left[\tfrac{1}{2} - \delta,\ \tfrac{1}{2} + \delta\right],
\]
we have either
\[
\left(f + \sum_{i,j} s^i_j \chi_{Q^i_j} + \sum_{i,j} t^i_j \chi_{W^i_j}\right)(A \cup C) \not\subset \left[\tfrac{1}{2} - \delta,\ \tfrac{1}{2} + \delta\right],
\]
in which case there is nothing to show, or
\[
\left(f + \sum_{i,j} s^i_j \chi_{Q^i_j} + \sum_{i,j} t^i_j \chi_{W^i_j}\right)(A \cup C) \subset \left[\tfrac{1}{2} - \delta,\ \tfrac{1}{2} + \delta\right],
\]
in which case we still have
\[
\left(f + \sum_{i,j} s^i_j \chi_{Q^i_j} + \sum_{i,j} t^i_j \chi_{W^i_j}\right) \geq 1 - 2\delta
\]
in a neighborhood of $1/2$, since $\chi_{Q^i_j}$ and $\chi_{W^i_j}$ have support disjoint from $J_n$.

We repeat the same argument for
$\operatorname{span}\big(X_n, \{\chi_{Q^i_j}\}_{i,j}, \{\chi_{W^i_j}\}_{i,j}\big)$ and $p_2$, and
continue inductively until $p_w$, to obtain a subspace $X_n'$ of $Y$ such
that (1) and (2) are fulfilled. Indeed, (2) is fulfilled for $X_n'$ by the previous argument, and to see that (1) still holds, we take $v, e \in S_{X_n}$ and denote by $\chi_{Q^{i,k}_j}$ and $\chi_{W^{i,k}_j}$ the characteristic functions obtained in the $k$-th step, where $\chi_{Q^{i,1}_j} = \chi_{Q^i_j}$ and $\chi_{W^{i,1}_j} = \chi_{W^i_j}$. Let $p_r \in \{p_1, \dots, p_w\}$ be such that $\|p_r - v\| < \varepsilon_n/2$. By the previous argument, we have scalars $\lambda_i'', \mu_i'' \in \mathbb{R}$ such that
\[
\left\|
p_r - \frac{1}{\left\lfloor 4/\varepsilon_n \right\rfloor + 1}
\sum_{j=1}^{\left\lfloor 4/\varepsilon_n \right\rfloor + 1}
\left(
p_r + \sum_{i=1}^{m_r} \lambda_i'' \, \chi_{Q^{i,r}_j} + \sum_{i=1}^{m_r} \mu_i'' \, \chi_{W^{i,r}_j}
\right)
\right\|
< \frac{\varepsilon_n}{2},
\]
while the sum is codirected to $e$ and equal to one. By the triangle inequality, we get
\[
\left\|
v - \frac{1}{\left\lfloor 4/\varepsilon_n \right\rfloor + 1}
\sum_{j=1}^{\left\lfloor 4/\varepsilon_n \right\rfloor + 1}
\left(
p_r + \sum_{i=1}^{m_r} \lambda_i'' \, \chi_{Q^{i,r}_j} + \sum_{i=1}^{m_r} \mu_i'' \, \chi_{W^{i,r}_j}
\right)
\right\|
< \varepsilon_n.
\]
Therefore $(1)$ will hold for $(X_n,X_{n+1})$, where $X_{n+1}$ will be defined shortly as some enlargement of $X_n'$. Now we define
\[
q_{n+1} := \frac12 \chi_{A \cup C} + \frac{1}{2}\,\chi_{B \setminus \left[\frac12-\frac{1}{20(n+1)},\, \frac12+\frac{1}{20(n+1)}\right]}
+ \chi_{\left[\frac12-\frac{1}{20(n+1)},\, \frac12+\frac{1}{20(n+1)}\right]} \in B(f_1,1/2) \cap B(f_2,1/2).
\]
We need to show that $X_{n+1} := \operatorname{span}(X_n', q_{n+1}) \subset Y$ still fulfills (2). Let $f_n \in X_n'$ and $\lambda \in \mathbb{R}$. Either we have $(f_n + \lambda q_{n+1})(A \cup C) \not\subset \left[\tfrac12-\delta,\ \tfrac12+\delta\right]$, in which case we have nothing to show, or $(f_n + \lambda q_{n+1})(A \cup C) \subset \left[\tfrac12-\delta,\ \tfrac12+\delta\right]$. Since $q_{n+1} \equiv 1/2 \equiv 1/2(\chi_{A \cup B} + \chi_{B \cup C})$ on $A \cup C$, we have $(f_n + \lambda \tfrac12(\chi_{A \cup B} + \chi_{B \cup C}))(A \cup C) \subset \left[\tfrac12-\delta,\ \tfrac12+\delta\right]$, and therefore, since (2) holds for $\operatorname{span}(\chi_{A \cup B} + \chi_{B \cup C}, f_n) \subset X_n'$, we get $f_n + \lambda \tfrac12(\chi_{A \cup B} + \chi_{B \cup C}) \geq 1-2\delta$ in a neighborhood of $1/2$. Therefore, since $q_{n+1} \equiv 1 = \tfrac12(\chi_{A \cup B} + \chi_{B \cup C})$ in a neighborhood of $1/2$, we get
\[
f_n + \lambda q_{n+1} \geq 1- 2\delta \quad \text{in a neighborhood of } 1/2,
\]
hence $(2)$ is fulfilled. The induction is therefore completed. $X:= \overline{\bigcup_{n=1}^\infty X_n} \subset Y$ has the Daugavet property. Indeed, it is a standard argument to show that the Daugavet property goes from a normed space to its completion, and the Daugavet property for $\bigcup_{n=1}^\infty X_n$ follows from (1) and the fact that any two elements in this union both lie in some $X_d$.
We can also see that it almost fulfills (2): For $u \in X$ with $u(A \cup C) \subset [1/2-\delta,1/2+\delta]$ and $\varepsilon > 0$, choose $u_\varepsilon \in X_{n_\varepsilon}$ with $\|u - u_\varepsilon\| < \varepsilon$. Define
\[
w_\varepsilon := \frac12\big(\chi_{A \cup B} + \chi_{B \cup C}\big) + \frac{\delta}{\delta+\varepsilon}\left(u_\varepsilon - \frac12\big(\chi_{A \cup B} + \chi_{B \cup C}\big)\right) \in X_{n_\varepsilon}.
\]
Since $\chi_{A \cup B} + \chi_{B \cup C} \equiv 1$ on $A \cup C$, we have $\tfrac12(\chi_{A\cup B}+\chi_{B\cup C}) \equiv 1/2$ on $A \cup C$. As $\|u-u_\varepsilon\|<\varepsilon$, we get $u_\varepsilon(A\cup C) \subset [1/2-\delta-\varepsilon,1/2+\delta+\varepsilon]$, and therefore
\[
w_\varepsilon(A \cup C) - \tfrac12 = \frac{\delta}{\delta+\varepsilon}\Big(u_\varepsilon(A\cup C) - \tfrac12\Big) \subset [-\delta,\delta],
\]
i.e., $w_\varepsilon(A \cup C) \subset [1/2-\delta,1/2+\delta]$. Since $w_\varepsilon \in X_{n_\varepsilon}$, (2) applied to $X_{n_\varepsilon}$ gives
\[
w_\varepsilon \geq 1-2\delta \quad \text{in a neighborhood of } 1/2.
\]
Obviously $w_\varepsilon \to u$, and since $w_\varepsilon$ is locally constant at $1/2$ for every $\varepsilon > 0$, we get, for $\varepsilon' > 0$ small enough that $\|w_{\varepsilon'} - u\| < \delta$, that $u \geq 1 - 3\delta$ in a neighborhood of $1/2$.
The set of all $u \in X$ such that $u(A \cup C) \subset [1/2-\delta,1/2+\delta]$ obviously contains $X \cap B(f_1,1/2 + \delta) \cap B(f_2,1/2 + \delta)$. Indeed, for $A$, we know that $\lvert f_1 - u \rvert \leq 1/2 + \delta$ implies $u \geq 1/2- \delta$, since $f_1 \equiv 1$ on $A$. Since $f_2 \equiv 0$ on $A$, we also have $\lvert 0 - u \rvert \leq 1/2 + \delta$, i.e., $u \leq 1/2 + \delta$ on $A$. The same argument obviously applies to $C$. Hence, for $v \in X \cap B(f_1,1/2 + \delta) \cap B(f_2,1/2 + \delta)$, we have $\|v + f_1\| \geq 2 - 3\delta$ by the last argument. Define $F \in S_{X^*}$ by
\[
F(f) := -\lim_{n \rightarrow \infty} \frac{n}{2}\int_{[1/2 -1/n,\, 1/2 + 1/n]} f\, dx.
\]
This is well-defined, since $f \in X$ can be approximated arbitrary well by functions in $\bigcup_{n = 1}^\infty X_n$ which are locally constant at $1/2$.
Let $g \in B_X$ with $F(g) > 1/2$ and $a > 0$. For $g_a \in \bigcup_{n = 1}^\infty X_n$ with $\mathcal{k}g_a - g \mathcal{k} < a$, we get $g_a < -1/2 + a$ in a neighborhood of $1/2$, and therefore $\|g + q_{n_a}\| \leq 1 + 1/2 + a$ for $n_a$ large enough. Therefore, the condition of Lemma 2.2 is not fulfilled for $\delta$ and $a$ small enough and hence $X$ does not have the WODP.
\end{proof}

The next theorem is a slight extension of Proposition~5.11 in~\cite{WODP2020}. The proof is a straightforward adaptation of the argument given there.

\begin{theorem}
Let $\varepsilon > 0$. 
Suppose $W \subset L[0,1]$ is relatively weakly compact, 
$\{x_i\}_{i=1}^{\infty} \subset S_{L[0,1]}$, 
and $S_i = S(B_{L^1[0,1]},g_i,\alpha_i)$ are slices of $B_{L[0,1]}$ with
$g_i \in S_{L^\infty[0,1]}$ and $\alpha_i > 0$. Then there exist an operator $T \in L(L[0,1])$ and a sequence $s_i \in S_i$ such that
\begin{enumerate}
    \item $\mathcal{k}T\mathcal{k} \le 1$,
    \item $\mathcal{k}Tw - w\mathcal{k} < \varepsilon \quad \text{for all } w \in W$,
    \item $Ts_i = x_i$ for all $i \in \mathbb{N}$.
\end{enumerate}
\end{theorem}

\begin{proof}

By the Dunford--Pettis characterization of relatively weakly compact
subsets of $L^1[0,1]$, there exists $\delta > 0$ such that for every
measurable set $B \subset [0,1]$ with $\mathcal{L}(B) < \delta$, we have
\[
  \int_B |w| \, dx < \frac{\varepsilon}{3}
  \qquad \text{for all } w \in W. \tag{$1'$}
\]

For $0 < \beta_i < \alpha_i$, choose measurable sets $A_i \subset [0,1]$
such that
\[
  A_i \subset \bigl\{ x \in [0,1] : |g_i(x)| \ge \|g_i\| - \beta_i \bigr\}
  \qquad \text{and} \qquad
  0 < \mathcal{L}(A_i) < \frac{\delta}{2^{2i}},
\]
and such that, for every $i \in \mathbb{N}$, the set
\[
  B_i := A_i \setminus \bigcup_{j=i+1}^{\infty} A_j
\]
has positive Lebesgue measure. Obviously, the sets $B_i$ are pairwise disjoint, and
\[
    \mathcal{L}\Big( \bigcup_{i=1}^\infty B_i \Big) < \delta.
\]
Define
\[
    s_i(x) := \operatorname{sign}(g_i(x)) \, \frac{\chi_{B_i}(x)}{\mathcal{L}(B_i)}.
\]
It follows that $s_i \in S_{L^1[0,1]} \cap S_i$, since
\[
    \int_{[0,1]} s_i g_i \, dx \geq \|g_i\| - \beta_i > 1 - \alpha_i.
\]

Denote $C := \bigcup_{i =1}^\infty B_i$. Choose $h_i \in S_{L^{\infty}([0,1])}$ such that
\[
    \int_{[0,1]} h_i s_i \, dx = 1
    \qquad \text{and} \qquad
    \operatorname{supp} h_i \subset B_i. \tag{2'}
\]
We can now define $T \in L(L^1[0,1])$ by
\[
    T(f) := f \chi_{[0,1] \setminus C}+ \sum_{i=1}^\infty \left( \int_{[0,1]} h_i f \, dx \right) x_i.
\]
The condition $(1')$ and $\operatorname{supp} h_i \subset B_i$ imply $\|T\| \leq 1$, and

\[
\|Tw - w\| \leq \|w\chi_{[0,1] \setminus C} - w\| + \sum_{i=1} \|w\chi_{B_i}\|\|x_i\| \leq 2\|w\chi_C\| \leq 2\varepsilon/3  < \varepsilon
\]
for $w \in W$. By $(2')$, we also see that $ Ts_i = x_i$ for every $i \in \mathbb{N}$.

\end{proof}

We isolate this property enjoyed by $L^1[0,1]$ in a new definition.

\begin{definition}
We say that a Banach space $X$ has the \emph{WCDOP} if, for every
relatively weakly compact set $W \subset X$, every $\varepsilon > 0$,
every slice $S$ of $B_X$, and every $x' \in B_X$, there exist $x \in S$
and an operator $T : X \to X$ such that
\begin{enumerate}
    \item $\|T\| \leq 1 + \varepsilon$,
    \item $\|T(w) - w\| < \varepsilon$ for every $w \in W$,
    \item $\|T(x) - x'\| < \varepsilon$.
\end{enumerate}
\end{definition}

The next theorem shows that the WCDOP and WODP do not coincide.
\begin{theorem}
$C[0,1]$ does not have the WCDOP.
\end{theorem}

\begin{proof}
Choose a sequence $w_n \in S_{C[0,1]}$ with $\operatorname{supp} w_n \subset (1/(n+1), 1/n)$ and $w_n \geq 0$. Since $(w_n)$ is equivalent to the $c_0$-basis, we conclude that $\{w_n\}_{n=1}^\infty$ converges weakly to $0$ and is therefore relatively weakly compact.

Now choose $S := \operatorname{Slice}(B_{C[0,1]}, \delta_0, 1/10)$, $f' \equiv -1$, and $\varepsilon = 1/10$. Suppose, for contradiction, that $C[0,1]$ has the WCDOP. Then there exist an operator $T \in L(C[0,1])$ with $\|T\| \leq 1 + 1/10$ and $\|Tw_n - w_n\| < 1/10$ for all $n$, and an element $f \in S$ such that
\[
    \frac{1}{10} > \|Tf - f'\| = \|Tf + 1\|.
\]
Therefore
\[
    \|w_n - f\| \geq \frac{1}{1 + 1/10} \|Tw_n - Tf\|
    \geq \frac{1}{1 + 1/10} \left( -\frac{2}{10} + \|w_n + 1\| \right)
    = \frac{2 - 2/10}{1 + 1/10},
\]
where the second inequality holds because
\[
    \|w_n + 1\| \leq \|w_n - Tw_n\| + \|Tw_n - Tf\| + \|Tf + 1\|.
\]
We conclude that, for every $n \in \mathbb{N}$, there is a point $x_n \in (1/(n+1), 1/n)$ such that $f(x_n) \leq -1/2$. Since $f$ is continuous, it follows that $f(0) \leq -1/2$, which contradicts $f \in S$.
\end{proof}

\section{Poor subspaces}

Independently of the last section, we turn to a short study of poor subspaces, answering two questions posed in \cite{quotient}.
The central tool for this investigation is Theorem~3.2, which is based on the theory, introduced in \cite{narrowrich}, of operators between Banach spaces $T \colon X \rightarrow Y$ that are narrow with respect to a pair $(G, \Gamma) \subset L(X,Y) \times \mathcal{P}(S_{X^*})$, i.e., for every $x \in S_X$, $y \in S_Y$, every $x^* \in \Gamma$, and every $\varepsilon > 0$, there exists $z \in S_X$ such that
\[
\|Gz + y\| > 2 - \varepsilon \quad \text{and} \quad \|T(x - z)\| + |x^*(x - z)| < \varepsilon.
\]
In the case of $G = \operatorname{Id}$ we say that $T$ is narrow with respect to $\Gamma$, and if also $\Gamma = S_{X^*}$ then we say that $T$ is narrow. For a Banach space $X$ with the Daugavet property we say that a closed subspace $Y \subset X$ is rich with respect to $\Gamma \subset S_{X^*}$ if the quotient map $q : X \rightarrow X/Y$ is narrow with respect to $\Gamma$, and rich if $\Gamma = S_{X^*}$. It was shown in~\cite[Theorem~5.12]{narrowrich} that a closed subspace $Y \subset X$ is rich if $Y \subset \tilde{Y} \subset X$ implies that $\tilde{Y}$ has the Daugavet property.

The following notion, which dualises the concept of rich subspaces, was introduced in \cite{quotient}.

\begin{definition}
Let $X \in \mathrm{DPr}$.  
A closed subspace $Z \subset X$ is said to be \emph{poor} if
\[
X / \tilde{Z} \in \mathrm{DPr}
\quad \text{for every closed subspace } \tilde{Z} \subset Z.
\]
\end{definition}

It turns out that a closed subspace of $Z \subset X$ is poor if the quotient map from $X^*$ to $X^*/Z^\perp$ is narrow with respect to $S_X$.

\begin{theorem}\cite[Theorem~5.8]{quotient} Let $X$ be a Banach space enjoying the Daugavet property. For a closed subspace $Z \subset X$, the following conditions are equivalent:
\begin{enumerate}
    \item $Z$ is poor.
    \item $X / \tilde{Z} \in \mathrm{DPr}$ for every closed subspace $\tilde{Z} \subset Z$ with
    $\mathrm{codim}_Z(\tilde{Z}) \le 2$.
    \item $Z^{\perp}$ is a subspace of $X^*$ that is rich with respect to $S_X$.
    \item For every $x^*, e^* \in S_{X^*}$, $\varepsilon > 0$, and for every $x \in S_X$ such that $e^*(x) > 1 - \varepsilon$,  
    there exists an element $v^* \in B_{X^*}$ satisfying
    \[
    v^*(x) > 1 - \varepsilon, \quad 
    \|x^* + v^*\| > 2 - \varepsilon, \quad 
    \|(e^* - v^*)|_Z\| < \varepsilon,
    \]
    that is, the quotient map from $X^*$ onto $X^* / Z^{\perp}$ is narrow with respect to $S_X$.
\end{enumerate}
\end{theorem}

It was shown in \cite[Theorem 6.10]{quotient} that in every $C(K)$-space with perfect metric compact $K$ and in every separable $L^1(\Omega, \Sigma, \mu)$-space with non-atomic $\mu$, there is a subspace isomorphic to $\ell^1$ that is not poor. By modifying the proof, we show that the statement holds in every separable Banach space with the Daugavet property.

\begin{theorem}
Let $X$ be a separable Banach space with the Daugavet property. Then $X$ contains a subspace isomorphic to $\ell^1$ which is not poor in $X$.
\end{theorem}

\begin{proof}
Since $X$ has the Daugavet property, we can find a sequence $\{u_n\} \subset S_X$ such that
\[
    \Big\| \sum_{n=1}^m a_n u_n \Big\| \geq \frac{9}{10} \sum_{n=1}^m |a_n|,
\]
for all $a_1, \dots, a_m \in \mathbb{R}$ and $m \in \mathbb{N}$. Let $\{x_n\} \subset S_X$ be a dense sequence in $S_X$. Then the sequence $v_n := u_n + \frac{8}{10} x_n$ is still equivalent to the unit vector basis of $\ell_1$, since
\[
    \Big\| \sum a_n v_n \Big\| \geq \Big\| \sum a_n u_n \Big\| - \frac{8}{10} \Big\| \sum a_n x_n \Big\|
    \geq \frac{9}{10} \sum |a_n| - \frac{8}{10} \sum |a_n| = \frac{1}{10} \sum |a_n|.
\]

Since Banach spaces with the Daugavet property are never isomorphic to $\ell^1$, $\overline{\operatorname{span}\{u_n\}}$ is a proper subspace of $X$; hence we may choose $e^* \in S_{X^*}$ such that $e^* \equiv 0$ on $\overline{\operatorname{span}\{u_n\}}$. Let $x^* := -e^*$, and fix a positive number $\varepsilon < 1/10$.

Suppose that $\overline{\operatorname{span}\{v_n\}}$ is poor. Then, by (4) in Theorem 3.2, we may find $v^* \in S_{X^*}$ such that
\[
    \|v^* - e^*\| = \|x^* + v^*\| > 2 - \varepsilon,
    \qquad
    \big\| (e^* - v^*)|_{\overline{\operatorname{span}\{v_n\}}} \big\| < \varepsilon.
    \tag{1}
\]
By the second inequality in (1), we have
\[
    \left| e^*\!\left( \tfrac{8}{10} x_n \right) - v^*(u_n + \tfrac{8}{10} x_n) \right|
    = \left| e^*(u_n + \tfrac{8}{10} x_n) - v^*(u_n + \tfrac{8}{10} x_n) \right| < \frac{2}{10},
    \tag{2}
\]
for every $n \in \mathbb{N}$. By the first inequality in (1), we may find some $x_i$ from our dense sequence such that $|(e^* - v^*)(x_i)| > 2 - \tfrac{1}{10}$, and hence, using (2),
\begin{align*}
    1 + \frac{6}{10}
    &> \frac{4}{10} + 1 + \left| e^*\!\left( \tfrac{8}{10} x_i \right) - v^*(u_i + \tfrac{8}{10} x_i) \right| \\
    &> \left| \tfrac{2}{10} (e^* - v^*)(x_i) + \left( e^*\!\left( \tfrac{8}{10} x_i \right) - v^*(u_i + \tfrac{8}{10} x_i) \right) + v^*(u_i) \right| \\
    &= \left| (e^* - v^*)(x_i) \right| \\
    &> 2 - \frac{1}{10}.
\end{align*}
This is a contradiction.
\end{proof}

The next definition was introduced in \cite{quotient}.

\begin{definition}
Let $X \in \mathrm{DPr}$. A subspace $E \subset X$ is said to be a \emph{bank} if $E$ contains an isomorphic copy of $\ell^1$ and for every poor subspace $Z$ of $X$, $q_E(B_Z)$ is nowhere dense in $B_{X/E}$ (here $q_E$ denotes the natural quotient map $q_E \colon X \to X/E$). If $E \subset X$ is a bank, then $B_{X/E}$ will be called the \emph{asset} of $E$.

\end{definition}

This is equivalent to saying that $E$ is a bank if it contains an isomorphic copy of $\ell_1$, while $\overline{Z + E} \neq X$ for every poor subspace $Z$ in $X$. The following was proved in \cite{quotient}, and it was asked whether the separability condition on the asset can be removed.

\begin{theorem}
Let $X \in \mathrm{DPr}$ and $E \subset X$ be a bank with separable asset. Then $X$ contains a copy of $\ell_1$ which is not poor in $X$.
\end{theorem}

The next theorem shows that the separability condition cannot be removed.

\begin{theorem}
$L^\infty[0,1]$ has a bank, and every separable subspace of it is poor in $L^\infty[0,1]$.
\end{theorem}

\begin{proof}

Let $X$ be a separable subspace of $L^\infty[0,1]$. Let $x^*, e^* \in S_{L^\infty[0,1]^*}$, $\varepsilon > 0$, and $f \in S_{L^\infty[0,1]}$ be such that $e^*(f) > 1 - \varepsilon$. Choose a dense sequence $\{x_n\}$ in $S_{X + \mathbb{R} f}$ that is closed under taking antipodes. Let
\[
    S := \{ q \in B_{L^\infty[0,1]} \mid x^*(q) > 1 - \varepsilon \}.
\]
Note that, by \cite[Proposition 5.4]{unendlichmunendlich}, or by an elementary consideration, there is some $g \in S$ such that
\[
    \|g + x_n\| > 2 - \varepsilon \quad \text{for every } n \in \mathbb{N}. 
\]
We can show by a standard argument that
\[
\|ag + bx_n\| \geq |a| + |b| - \varepsilon \max(\lvert a \lvert, \lvert b \lvert) \geq (|a| + |b|)(1- \varepsilon) \tag{1}
\]
for all $a, b \in \mathbb{R}$ and $n \in \mathbb{N}$. Indeed, since $\{x_n\}_{n \in \mathbb{N}}$ is closed under taking antipodes, we can assume $a, b \geq 0$. Let $a \geq b$. Then
\[
\|ag + bx_n\| \geq \|a(g + x_n) - (a-b)x_n\| \geq a\|g + x_n\| - (a-b)\|x_n\| \geq a(2-\varepsilon) - (a-b) \geq a + b - a\varepsilon.
\]

Now define $\mu \colon X + \mathbb{R}f + \mathbb{R}g \to \mathbb{R}$ by
\[
    \mu(h + rg) = e^*(h) + r,
\]
for $h \in X + \mathbb{R}f$ and $r \in \mathbb{R}$, which is well defined, since $g$ is linearly independent from $X + \mathbb{R}f$. By (1) and the density of $\{x_n\}$ in $S_{X + \mathbb{R}f}$, we have
\[
    |\mu(h + rg)| \leq \|h\| + |r| \leq \frac{1}{1 - \varepsilon} \|h + rg\|,
\]
and therefore $\|\mu\| \leq 1/(1 - \varepsilon)$.

Extend $\mu$ norm-preservingly to the whole space $X$ using the Hahn--Banach theorem. We have $\mu(g) = 1$, and therefore $\|\mu + x^*\| > 2 - \varepsilon$, since $g \in S$. Also, $\mu \equiv e^*$ on $X + \mathbb{R}f$. Therefore, $X$ is poor in $L^\infty[0,1]$.

Since $L^\infty[0,1]$ is isometrically isomorphic to $C(K)$, where $K$ is a perfect compact Hausdorff space, we get by \cite[Corollary 6.4]{quotient} that $L^\infty[0,1]$ has a bank.

\end{proof}

\section*{Acknowledgements}

The author would like to express his gratitude to Dirk Werner and Abraham Rueda Zoca for their valuable advice and corrections on parts of the paper.

\bibliographystyle{plain}
\bibliography{mybib}

\end{document}